\documentclass[12pt]{amsart}
\usepackage{amsmath,amsthm,amsfonts,amssymb,mathrsfs}
\usepackage{pb-diagram}
\date{\today}

\usepackage{color}
\input xymatrix
\xyoption{all}

\usepackage{hyperref}

\newtheorem{theorem}{Theorem}

\newtheorem{corollary}[theorem]{Corollary}
\newtheorem{lemma}[theorem]{Lemma}
\theoremstyle{definition}

\newcommand{\N}{\mathbb N}

\def\N{\mathbb N}
\def\t_c{\tau_{comp}}
\def\op{\operatorname}

\begin{document}

\title[An alternative look at the structure of graph inverse semigroups]{An alternative look at the structure of graph inverse semigroups}

\author[S.~Bardyla]{Serhii~Bardyla}
\thanks{The work of the author is supported by the Austrian Science Fund FWF (Grant  I
3709 N35).}
\address{S. Bardyla: Institute of Mathematics, Kurt G\"{o}del Research Center, Vienna, Austria}
\email{sbardyla@yahoo.com}



\begin{abstract}
For any graph inverse semigroup $G(E)$ we describe subsemigroups $D^0=D\cup\{0\}$ and $J^0=J\cup\{0\}$ of $G(E)$ where $D$ and $J$ are arbitrary $\mathcal{D}$-class and $\mathcal{J}$-class of $G(E)$, respectively. In particular, we prove that for each $\mathcal{D}$-class $D$ of a graph inverse semigroup over an acyclic graph the semigroup $D^0$ is isomorphic to a semigroup of matrix units. Also we show that for any elements $a,b$ of a graph inverse semigroup $G(E)$,
$J_a\cdot J_b\cup J_b\cdot J_a\subset J_b^0$ if there exists a path $w$ such that $s(w)\in J_a$ and $r(w)\in J_b$.
\end{abstract}
\maketitle
\section{Preliminaries}
We shall follow the terminology of~\cite{Clifford-Preston-1961-1967} and~\cite{Lawson-1998}. By $\N$ we denote the set of positive integers.
The cardinality of a set $X$ is denoted by $|X|$.
A semigroup $S$ is called an \emph{inverse semigroup} if for each element $a\in S$ there exists a unique
element $a^{-1}\in S$ such that
$aa^{-1}a=a$ and $a^{-1}aa^{-1}=a^{-1}$.

By $\mathcal{R}$,
$\mathcal{L}$, $\mathcal{J}$, $\mathcal{D}$ and $\mathcal{H}$ we denote
Green's relations on a semigroup $S$ which are defined as follows: for each $a,b\in S$
\begin{center}
\begin{tabular}{rcl}
  $a\mathcal{R}b$ & if and only if & $aS\cup\{a\}=bS\cup\{b\}$; \\
  $a\mathcal{L}b$ & if and only if & $Sa\cup\{a\}=Sb\cup\{b\}$; \\
  $a\mathcal{J}b$ & if and only if & $SaS\cup aS\cup Sa\cup \{a\}=SbS\cup bS\cup Sb\cup \{b\}$; \\
    & $\mathcal{D}=\mathcal{L}{\circ}\mathcal{R}=\mathcal{R}{\circ}\mathcal{L}$; & $\mathcal{H}=\mathcal{L}\cap\mathcal{R}$. \\
\end{tabular}
\end{center}

Let $S$ be a semigroup with zero $0_S$ and $X$ be a non-empty set. By $\mathcal{B}_{X}(S)$ we denote the set
 $
 X{\times}S{\times}X\sqcup\{0\}
 $
endowed with the following semigroup
operation:
\begin{equation*}
\begin{split}
&(a,s,b)\cdot(c,t,d)=
\left\{
  \begin{array}{cl}
    (a,s\cdot t,d), & \hbox{ if~ } b=c;\\
    0, & \hbox{ if~ } b\neq c,
  \end{array}
\right.\\
&\hbox{and } (a,s,b)\cdot 0=0\cdot(a,s,b)=0\cdot 0=0, \hbox{ for each } a,b,c,d\in X \hbox{ and } s,t\in S.
\end{split}
\end{equation*}
 The semigroup $\mathcal{B}_{X}(S)$ is called the \emph{Brandt $X$-extension of the semigroup $S$}. Obviously, the set $J=\{(a,0_S,b)\mid a,b\in X\}\cup\{0\}$ is a two-sided ideal of the semigroup $\mathcal{B}_{X}(S)$. The Rees factor semigroup $\mathcal{B}_{X}(S)/J$ is called the {\em Brandt $X^0$-extension of the semigroup $S$} and is denoted by $\mathcal{B}_{X}^0(S)$. If $S$ is the semilattice $(\{0,1\},\min)$ then we denote the semigroup $\mathcal{B}_{X}^0(S)$ by $\mathcal{B}^0_X$. The semigroup $\mathcal{B}^0_X$ is well-known (see page 86 from~\cite{Lawson-1998}) and is called the {\em semigroup of $X{\times}X$-matrix units}. Observe that semigroups $\mathcal{B}^0_{X}(S)$ and $\mathcal{B}^0_{Y}(S)$ are isomorphic iff $|X|=|Y|$. A Brandt $X^0$-extension of a group play an important role in the structure of primitive inverse semigroups (see \cite[Chapter~3.3]{Lawson-1998}). Algebraic and topological properties of a Brandt $X^0$-extension of a semigroup were investigated in~\cite{Gutik-Pavlyk-2006} and~\cite{GutRep-2010}.



For a cardinal $\lambda$ {\em polycyclic monoid} $\mathcal{P}_\lambda$ is the semigroup with identity $1$ and zero $0$ given by the presentation:
\begin{equation*}
    \mathcal{P}_\lambda=\left\langle 0, 1, \left\{p_i\right\}_{i\in\lambda}, \left\{p_i^{-1}\right\}_{i\in\lambda}\mid  p_i^{-1}p_i=1, p_j^{-1}p_i=0 \hbox{~for~} i\neq j\right\rangle.
\end{equation*}

Observe that polycyclic monoid $\mathcal{P}_0$ is isomorphic to the semilattice $(\{0,1\},\min)$.
Polycyclic monoid is a generalization of the well-known bicyclic monoid (see~\cite[Chapter~3.4]{Lawson-1998}). More precisely, the bicyclic monoid with adjoined zero is isomorphic to the polycyclic monoid $\mathcal{P}_{1}$.
Polycyclic monoid $\mathcal{P}_{k}$ over a finite non-zero cardinal $k$ was introduced in \cite{Nivat-Perrot-1970}.
Algebraic and topological properties of polycyclic monoids were investigated
in~\cite{Bardyla-2016(1),BardGut-2016(1),BardGut-2016(2),Gutik-2015,Lawson-2009,Meakin-1993}.

A {\em directed} graph $E=(E^{0},E^{1},r,s)$ consists of disjoint sets $E^{0},E^{1}$ of {\em vertices} and {\em edges}, respectively, together with functions $s,r:E^{1}\rightarrow E^{0}$ which are called {\em source} and {\em range}, respectively. In this paper we refer to a directed graph simply as ``graph". We consider each vertex as a path of length zero. A path of non-zero length $x=e_{1}\ldots e_{n}$ in a graph $E$ is a finite sequence of edges $e_{1},\ldots,e_{n}$ such that $r(e_{i})=s(e_{i+1})$ for each positive integer $i<n$. By $\operatorname{Path}^+(E)$ we denote the set of all paths of a graph $E$ which have a non-zero length. We extend functions $s$ and $r$ on the set $\operatorname{Path}(E)=E^0\cup \operatorname{Path}^+(E)$ of all paths in the graph $E$ as follows: for each vertex $e\in E^0$ put $s(e)=r(e)=e$ and for each path of non-zero length $x=e_{1}\ldots e_{n}\in \operatorname{Path}^+(E)$ put $s(x)=s(e_{1})$ and $r(x)=r(e_{n})$. By $|x|$ we denote the length of a path $x$. Let $a=e_1\ldots e_n$ and $b=f_1\ldots f_m$ be two paths such that $|a|>0$, $|b|>0$ and $r(a)=s(b)$. Then by $ab$ we denote the path $e_1\ldots e_nf_1\ldots f_m$.
If $a$ is a vertex and $b$ is a path such that $s(b)=a$ ($r(b)=a$, resp.) then put $ab=b$ ($ba=b$, resp.).
A path $x$ is called a {\em prefix} of a path $y$ if there exists a path $z$ such that $y=xz$. An edge $e$ is called a {\em loop} if $s(e)=r(e)$. A path $x$ is called a {\em cycle} if $s(x)=r(x)$ and $|x|>0$. Vertices $a$ and $b$ of a graph $E$ are called {\em strongly connected} if there exist paths $u,v\in \op{Path}(E)$ such that $a=s(u)=r(v)$ and $b=s(v)=r(u)$. Define a relation $R$ on the set $E^0$ as follows: $(a,b)\in R$ iff vertices $a$ and $b$ are strongly connected. Simple verifications show that $R$ is an equivalence relation. Equivalence classes of the relation $R$ are called {\em strongly connected components} of a graph $E$.
A graph $E$ is called {\em acyclic} if it contains no cycles.

For a given directed graph $E=(E^{0},E^{1},r,s)$ a graph inverse semigroup (or simply GIS) $G(E)$ over the graph $E$ is the semigroup with zero generated by the sets $E^{0}$, $E^{1}$ together with the set $E^{-1}=\{e^{-1}|\hbox{ } e\in E^{1}\}$ which is disjoint with $E^0\cup E^1$ satisfying the following relations for all $a,b\in E^{0}$ and $e,f\in E^{1}$:
\begin{equation*}
\begin{split}
&(1)\quad  a\cdot b=a \hbox{ if } a=b \hbox{ and } a\cdot b=0 \hbox{ if } a\neq b;\\
&(2)\quad s(e)\cdot e=e\cdot r(e)=e;\\
&(3)\quad e^{-1}\cdot s(e)=r(e)\cdot e^{-1}=e^{-1};\\
&(4)\quad e^{-1}\cdot f=r(e) \hbox{ if } e=f \hbox{ and } e^{-1}\cdot f=0 \hbox{ if } e\neq f.
\end{split}
\end{equation*}

Graph inverse semigroups are generalizations of the polycyclic monoids. In particular, for each cardinal $\lambda$ a polycyclic monoid $\mathcal{P}_{\lambda}$ is isomorphic to the graph inverse semigroup over the graph $E$ which consists of one vertex and $\lambda$ distinct loops.
However, by~\cite[Theorem~1]{Bardyla-2017(2)}, each graph inverse semigroup $G(E)$ is isomorphic to a subsemigroup of the polycyclic monoid $\mathcal{P}_{|G(E)|}$.

According to~\cite[Chapter~3.1]{Jones-2011}, each non-zero element of a graph inverse semigroup $G(E)$ can be uniquely represented as $uv^{-1}$ where $u,v\in \operatorname{Path}(E)$ and $r(u)=r(v)$. A semigroup operation in $G(E)$ is defined by the following way:
\begin{equation*}
\begin{split}
  &  u_1v_1^{-1}\cdot u_2v_2^{-1}=
    \left\{
      \begin{array}{ccl}
        u_1wv_2^{-1}, & \hbox{if~~} u_2=v_1w & \hbox{for some~} w\in \operatorname{Path}(E);\\
        u_1(v_2w)^{-1},   & \hbox{if~~} v_1=u_2w & \hbox{for some~} w\in \operatorname{Path}(E);\\
        0,              & \hbox{otherwise},
      \end{array}
    \right.\\
  &  \hbox{and } uv^{-1}\cdot 0=0\cdot uv^{-1}=0\cdot 0=0.
    \end{split}
\end{equation*}

Further, when we write an element of $G(E)$ in a form $uv^{-1}$ we always mean that $u,v\in\operatorname{Path}(E)$ and $r(u)=r(v)$.
Simple verifications show that $G(E)$ is an inverse semigroup and $(uv^{-1})^{-1}=vu^{-1}$, for each element $uv^{-1}\in G(E)$.

Graph inverse semigroups play an important role in the study of rings and $C^{*}$-algebras (see \cite{Abrams-2005,Ara-2007,Cuntz-1980,Kumjian-1998,Paterson-1999}).
Algebraic theory of graph inverse semigroups is well developed (see~\cite{Amal-2016, Bardyla-2017(2), Jones-2011, Jones-Lawson-2014, Lawson-2009, Mesyan-2016}).
Topological properties of graph inverse semigroups were investigated
in~\cite{Bardyla-2017(1),Bardyla-2018(1),Bardyla-2018,Mesyan-Mitchell-Morayne-Peresse-2013}.

This paper is inspired by the paper of Mesyan and Mitchell~\cite{Mesyan-2016} and can be regarded as an alternative look at the structure of graph inverse semigroups.

\section{A local structure of graph inverse semigroups}
By~\cite[Corollary~2]{Mesyan-2016}, two non-zero elements $ab^{-1}$ and $cd^{-1}$ of a GIS $G(E)$ are $\mathcal{D}$-equivalent iff $r(a)=r(b)=r(c)=r(d)$. Observe that each non-zero $\mathcal{D}$-class contains exactly one vertex of $E$. By $D_e$ we denote the $\mathcal{D}$-class which contains vertex $e\in E^0$. Put $D_e^0=D_e\cup\{0\}$.
\begin{lemma}\label{lemma0}
Let $G(E)$ be a GIS, $ab^{-1}\in D_e^0$ and $cd^{-1}\in D_f^0$. Then
$ab^{-1}\cdot cd^{-1}\in D_e^0\cup D_f^0$.
\end{lemma}

\begin{proof}
Fix any elements $ab^{-1}\in D_e^0$ and $cd^{-1}\in D_f^0$. If $ab^{-1}\cdot cd^{-1}=0$ then $ab^{-1}\cdot cd^{-1}\in D_e^0\cup D_f^0$. Assume that $ab^{-1}\cdot cd^{-1}\neq 0$. Then there exists a path $w\in\operatorname{Path}(E)$ such that either $ab^{-1}\cdot cd^{-1}=awd^{-1}$ or $ab^{-1}\cdot cd^{-1}=a(dw)^{-1}$. In the first case $c=bw$ which yields that $r(aw)=r(w)=r(c)=r(d)=f$. Hence $ab^{-1}\cdot cd^{-1}\in D_f$. In the second case $b=cw$ which implies that $r(dw)=r(w)=r(b)=r(a)=e$. Hence $ab^{-1}\cdot cd^{-1}\in D_e$.
\end{proof}

Observe that if $uv^{-1}\in D_e$ then $(uv^{-1})^{-1}=vu^{-1}\in D_e$ which provides the following:
\begin{corollary}\label{corollary1}
For each vertex $e$ of a graph $E$, $D_e^0$ is an inverse subsemigroup of $G(E)$.
\end{corollary}

Further we need the following denotations. For any vertex $e$ of a graph $E$ put:
$$I_e=\{u\in \op{Path}(E)\mid r(u)=e\};$$
$$Q_e=\{u\in I_e\mid r(v)\neq e, \hbox{ for each non-trivial prefix }v \hbox{ of }u\};$$
$$C_e=\{u\in I_e\mid s(u)=r(u)=e\};$$
$$C^1_e=C_e\cap Q_e=\{u\in C_e\mid r(v)\neq e,\hbox{ for each non-trivial prefix } v \hbox{ of } u\}.$$
By $\langle C_e\rangle$ ($\langle C^1_e\rangle$, resp.) we denote the inverse subsemigroup of $G(E)$ which is generated by the set $C_e\cup\{0\}$
($C^1_e\cup\{0\}$, resp.). Observe that $e\in C^1_e$ and $e$ is the identity of the semigroup $\langle C_e\rangle$.
The following theorem describes the structure of the semigroup $\langle C_e\rangle$.
\begin{theorem}\label{poly}
For each vertex $e$ of any graph $E$ the semigroup $\langle C_e\rangle$ is isomorphic to the polycyclic monoid $\mathcal{P}_{|C_e^1\setminus\{e\}|}$.
\end{theorem}

\begin{proof}
Fix any vertex $e\in E^0$.
Put $\lambda=|C^1_e\setminus\{e\}|$. Let $C^1_e\setminus\{e\}=\{u_{\alpha}\}_{\alpha\in \lambda}$. For convenience we denote $e$ by $u_{-1}$.
It is easy to check that
$\langle C^1_e\rangle=\{uv^{-1}\mid u,v\in C_e\}\cup\{0\}=\langle C_e\rangle.$


Let $G=\{p_{\alpha}\}_{\alpha\in \lambda}\cup\{p_{\alpha}^{-1}\}_{\alpha\in \lambda}$ be the set of generators of the polycyclic monoid $\mathcal{P}_{\lambda}$.
Define a map $f: \langle C_e\rangle \rightarrow \mathcal{P}_{\lambda}$ consecutively extending it as follows. At first we define $f$ on $C^1_e$ by putting $f(u_{-1})=1$ and $f(u_{\alpha})=p_{\alpha}$ for each $\alpha\in\lambda$. Let $u\in C_e\setminus\{e\}$ be any element. It is easy to check that $u$ has a unique representation $u=u_{\alpha_1}u_{\alpha_2}\ldots u_{\alpha_n}$ as a product of elements of $C^1_e\setminus\{e\}$. We put  $f(u)=p_{\alpha_{1}}p_{\alpha_{2}}\ldots p_{\alpha_{n}}$. Observe that any non-zero element of $\langle C_e\rangle$ has a unique representation in a form $uv^{-1}$ for some paths $u,v\in C_e$. Finally, put $f(uv^{-1})=f(u)f(v)^{-1}$ and $f(0)=0$.


Obviously, $f$ is a bijection. Let us show that $f$ is a homomorphism. Fix any elements $ab^{-1},cd^{-1}\in \langle C_e\rangle$. Let
$$a=u_{\alpha_1}\ldots u_{\alpha_n},\quad b=u_{\beta_1}\ldots u_{\beta_{m}},\quad c=u_{\gamma_1},\ldots u_{\gamma_k},\quad d=u_{\delta_1}\ldots u_{\delta_t}$$
be (unique) representations of elements $a,b,c,d$ as a product of elements of $C^1_e$ (here we agree that if some of the elements $a,b,c$ or $d$ are equal to $e$, then their representations are equal to $u_{-1}$). There are three cases to consider:
\begin{itemize}
\item[$(1)$] $b$ is a prefix of $c$;
\item[$(2)$] $c$ is a prefix of $b$;
\item[$(3)$] $ab^{-1}\cdot cd^{-1}=0$.
\end{itemize}
Suppose that case $1$ holds, i.e., $u_{\gamma_1}\ldots u_{\gamma_k}=u_{\beta_1}\ldots u_{\beta_{m}}u_{\gamma_{m+1}}\ldots u_{\gamma_{k}}$. Observe that
$$ab^{-1}\cdot cd^{-1}=u_{\alpha_1}\ldots u_{\alpha_n}u_{\gamma_{m+1}}\ldots u_{\gamma_{k}}(u_{\delta_1}\ldots u_{\delta_t})^{-1}.$$
Then
\begin{equation*}
f(ab^{-1}\cdot cd^{-1})=f(u_{\alpha_1}\ldots u_{\alpha_n}u_{\gamma_{m+1}}\ldots u_{\gamma_{k}}(u_{\delta_1}\ldots u_{\delta_t})^{-1})=p_{\alpha_1}\ldots p_{\alpha_n}p_{\gamma_{m+1}}\ldots p_{\gamma_{k}}(p_{\delta_1}\ldots p_{\delta_t})^{-1}.
\end{equation*}
On the other hand,
\begin{equation*}
\begin{split}
&f(ab^{-1})\cdot f(cd^{-1})=p_{\alpha_1}\ldots p_{\alpha_n}(p_{\beta_m}^{-1}\ldots p_{\beta_1}^{-1}\cdot p_{\beta_1}\ldots p_{\beta_{m}})p_{\gamma_{m+1}}\ldots p_{\gamma_{k}} (p_{\delta_1}\ldots p_{\delta_t})^{-1}=\\
&= p_{\alpha_1}\ldots p_{\alpha_n}p_{\gamma_{m+1}}\ldots p_{\gamma_{k}}(p_{\delta_1}\ldots p_{\delta_t})^{-1}=f(ab^{-1}\cdot cd^{-1}).
\end{split}
\end{equation*}

Case $2$ is similar to case $1$. Consider case $3$. In this case there exists a positive integer $i$ such that $u_{\beta_j}=u_{\gamma_j}$ for every $j<i$ and $u_{\beta_i}\neq u_{\gamma_i}$. Observe that $f(ab^{-1}\cdot cd^{-1})=f(0)=0$ and
\begin{equation*}
\begin{split}
&f(ab^{-1})\cdot f(cd^{-1})= p_{\alpha_1}\ldots p_{\alpha_n}p_{\beta_m}^{-1}\ldots p_{\beta_i}^{-1}(p_{\beta_{i-1}}^{-1}\ldots p_{\beta_{1}}^{-1}\cdot p_{\beta_1}\ldots p_{\beta_{i-1}})p_{\gamma_{i}}\ldots p_{\gamma_{k}}(p_{\delta_1}\ldots p_{\delta_t})^{-1}=\\
&=p_{\alpha_1}\ldots p_{\alpha_n}p_{\beta_m}^{-1}\ldots (p_{\beta_i}^{-1}\cdot p_{\gamma_{i}})\ldots p_{\gamma_{k}} (p_{\delta_1}\ldots p_{\delta_t})^{-1}=0=f(ab^{-1}\cdot cd^{-1}).\\
\end{split}
\end{equation*}
Hence $f$ is an isomorphism between semigroups $\langle C_e\rangle$ and $\mathcal{P}_{|C_e^1\setminus\{e\}|}$.
\end{proof}

The following theorem describes the structure of a subsemigroup $D_e^0$ of an arbitrary GIS.
\begin{theorem}\label{D}
Let $E$ be any graph and $e\in E^0$. Then the semigroup $D_e^0$ is isomorphic to the Brandt $Q_e^0$-extension of the polycyclic monoid $\mathcal{P}_{|C_e^1\setminus\{e\}|}$. 
\end{theorem}
\begin{proof}
Recall that $D_e=\{uv^{-1}\mid r(u)=r(v)=e\}$. The proof is based on the following obvious fact: each element $u\in I_e$ can be uniquely represented as follows: $u=u_1u_2$ where $u_1\in Q_e$ and $u_2\in C_e$ (here both $u_1$ and $u_2$ can be equal to $e$). By Theorem~\ref{poly}, the semigroup $\langle C_e\rangle$ is isomorphic to the polycyclic monoid $\mathcal{P}_{|C_e^1\setminus\{e\}|}$. Define the map $h:D_e^0\rightarrow B_{Q_e}^0(\mathcal{P}_{|C_e^1\setminus\{e\}|})$ as follows: $h(0)=0$ and $h(uv^{-1})=(u_1,f(u_2v_2^{-1}),v_1)$ for each non-zero element $uv^{-1}=u_1u_2(v_1v_2)^{-1}\in D_e$ where $u_1,v_1\in Q_e$, $u_2,v_2\in C_e$ and $f$ is an isomorphism between semigroups $\langle C_e\rangle$ and $\mathcal{P}_{|C_e^1\setminus\{e\}|}$ defined in Theorem~\ref{poly}. We remark that $f(e)=(e,1,e)$. Suppose that $u\neq v$ for some paths $u,v\in I_e$. Then $u_1\neq v_1$ or $u_2\neq v_2$ which implies that the map $h$ is injective. Since for each non-zero element $(a,uv^{-1},b)\in B_{Q_e}^0(\mathcal{P}_{|C_e^1\setminus\{e\}|})$
we have
$$h(af^{-1}(u)(bf^{-1}(v))^{-1})=(a,f(f^{-1}(u)f^{-1}(v)^{-1}),b)=(a,ff^{-1}(uv^{-1}),b)=(a,uv^{-1},b)$$
the map $h$ is bijective. Now it remains to show that $h$ is a homomorphism. Fix any elements $ab^{-1},cd^{-1}\in D_e$. Following the main idea of the proof we can uniquely represent elements $a,b,c,d\in I_e$ as follows: $a=a_1a_2, b=b_1b_2, c=c_1c_2$ and $d=d_1d_2$ where $a_1,b_1,c_1,d_1\in Q_e$ and $a_2,b_2,c_2,d_2\in C_e$. 
There are three cases to consider:
\begin{itemize}
\item[$(1)$] There exists $w\in\op{Path}(E)$ such that $ab^{-1}\cdot cd^{-1}=awd^{-1}$, i.e., $c=bw$;
\item[$(2)$] there exists $w\in\op{Path}(E)$ such that $ab^{-1}\cdot cd^{-1}=a(dw)^{-1}$, i.e., $b=cw$;
\item[$(3)$] $ab^{-1}\cdot cd^{-1}=0$.
\end{itemize}
Consider case $1$. Observe that $s(w)=r(b)=r(c)=r(w)$ which implies that $c_1=b_1$ and $c_2=b_2w$. Hence
\begin{equation*}
\begin{split}
&h(ab^{-1})\cdot h(cd^{-1})=(a_1,f(a_2b_2^{-1}),b_1)\cdot (b_1,f(b_2wd_2^{-1}),d_1)=(a_1,f(a_2b_2^{-1})\cdot f(b_2wd_2^{-1}), d_1)=\\
&=(a_1,f(a_2b_2^{-1}\cdot b_2wd_2^{-1}), d_1)=(a_1,f(a_2wd_2^{-1}), d_1)=h(awd^{-1}).
\end{split}
\end{equation*}
Consider case $2$. Observe that $s(w)=r(c)=r(b)=r(w)$ which implies that $c_1=b_1$ and $b_2=c_2w$. Similar calculations as in case $1$ show that $h(ab^{-1})\cdot h(cd^{-1})=h(a(dw)^{-1})$.

Consider case $3$. Observe that for each path $w\in \op{Path}(E)$ neither $b=cw$ nor $c=bw$. Then one of the following two subcases holds:
\begin{itemize}
\item[$(3.1)$] $b_1\neq c_1$;
\item[$(3.2)$] $b_1=c_1$, but for any $w\in\op{Path}(E)$ neither $b_2=c_2w$ nor $c_2=b_2w$.
\end{itemize}
Consider subcase $3.1$. Then
$$
h(ab^{-1})\cdot h(cd^{-1})=(a_1,f(a_2b_2^{-1}),b_1)\cdot (c_1,f(c_2d_2^{-1}),d_1)=0=f(0).
$$
Consider subcase $3.2$. Then
\begin{equation*}
\begin{split}
&h(ab^{-1})\cdot h(cd^{-1})=(a_1,f(a_2b_2^{-1}),b_1)\cdot (b_1,f(c_2d_2^{-1}),d_1)=(a_1,f(a_2b_2^{-1})\cdot f(c_2d_2^{-1}), d_1)=\\
&=(a_1,f(a_2b_2^{-1}\cdot c_2d_2^{-1}), d_1)=(a_1,0, d_1)=0=h(0).
\end{split}
\end{equation*}
Hence the map $h$ is an isomorphism between semigroups $D_e^0$ and $B_{Q_e}^0(\mathcal{P}_{|C_e^1\setminus\{e\}|})$.
\end{proof}

Graph $E$ is called {\em acyclic at a vertex} $e\in E^0$ if $C_e=\{e\}$.
\begin{corollary}\label{matrix}
Let $E$ be a graph which is acyclic at a vertex $e$. Then the subsemigroup $D_e^0$ of $G(E)$ is isomorphic to the semigroup of $I_e{\times}I_e$-matrix units $\mathcal{B}_{I_e}^0$.
\end{corollary}
\begin{proof}
Recall that by $B_X^0$ we denote the semigroup $B_X^0(\mathcal{P}_0)$.
Since graph $E$ is acyclic at a vertex $e$ we obtain that $I_e=Q_e$ and $C_e=\{e\}$. By Theorem~\ref{D}, the semigroup $D_e^0$ is isomorphic to the semigroup $B_{I_e}^0(\mathcal{P}_{0})$.
\end{proof}

By~\cite[Corollary~2]{Mesyan-2016}, two non-zero elements $ab^{-1}$ and $cd^{-1}$ of a GIS $G(E)$ are $\mathcal{J}$-equivalent iff there exist elements $u,v\in \op{Path}(E)$ such that $s(u)=r(a)=r(v)$ and $r(u)=r(c)=s(v)$. There exists a one to one correspondence between the set of strongly connected components of a graph $E$ and non-zero $\mathcal{J}$-classes of a GIS $G(E)$. More precisely, $J\cap E^0$ is a strongly connected component of a graph $E$ for each non-zero $\mathcal{J}$-class $J$ of $G(E)$. Therefore, by $J_A$ we denote a $\mathcal{J}$-class which contains a strongly connected component $A\subset E^0$.

Observe that for each strongly connected component $A$ of a graph $E$, $\mathcal{J}_A=\cup_{e\in A}D_e$. Hence Lemma~\ref{lemma0} provides the following:
\begin{corollary}\label{J}
For each strongly connected component $A$ of a graph $E$ the set $J_A^0=J_A\cup\{0\}$ is an inverse subsemigroup of $G(E)$.
\end{corollary}

Let $E$ be a graph and $X$ be any non-empty subset of $E^0$. By $E_X$ we denote the induced (by the set $X$) subgraph of the graph $E$, i.e., $E_X^0=X$, $E_X^1=\{x\in E^1\mid s(x)\in X\hbox{ and } r(x)\in X\}$ and source (resp., range) function $s_X$ (resp., $r_X$) of the graph $E_X$ is the restriction of the source function $s$ (resp., range function $r$) of the graph $E$ on the set $E_X^1$. Let $A$ be a strongly connected component of a graph $E$. Put
$$I_A=\{u\in \op{Path}(E)\mid r(u)\in A\};$$
$$Q_A=\{u\in I_A\mid r(v)\notin A, \hbox{ for each non-trivial prefix }v \hbox{ of }u\}.$$

\begin{lemma}\label{C}
Let $A$ be a strongly connected component of a graph $E$ and $w$ be a path such that $s(w)\in A$ and $r(w)\in A$. Then $w\in \op{Path}(E_A)$ where $E_A$ is an induced subgraph of $E$.
\end{lemma}
\begin{proof}
The proof is obvious if $w=e\in A$.
Let $w=a_1\ldots a_n$ be a path of non-zero length such that $s(w)=s(a_1)=e_1\in A$ and $r(w)=r(a_n)=f\in A$. Put $s(a_i)=e_i$, for each $i\leq n$.
Since vertices $e_1$ and $f$ belong to $A$ there exists a path $u$ such that $s(u)=f$ and $r(u)=e_1$. We claim that for each $i\leq n$ vertices $e_i$ belong to $A$. Indeed, put $x=a_1\ldots a_{i-1}$ and $y=a_i\ldots a_nu$. Then $s(x)=e_1, r(x)=e_i$ and $s(y)=e_i, r(y)=e_1$ which provides that $\{e_i\}_{i\leq n}\subseteq A$. Hence $a_i\in E_A^1$ for every $i\leq n$ and, as a consequence, $w\in \op{Path}(E_A)$.
\end{proof}

The following theorem describes the structure of a subsemigroup $J_A^0$ of $G(E)$ where $A$ is any strongly connected component of a graph $E$.
\begin{theorem}\label{JJ}
Let $E$ be any graph and $A\subseteq E^0$ be a strongly connected component. Then the semigroup $J_A^0$ is isomorphic to a subsemigroup of the Brandt $Q_A^0$-extension of the graph inverse semigroup $G(E_A)$ over the induced subgraph $E_A$.
\end{theorem}
\begin{proof}
The proof of this theorem is based on the following fact which follows from Lemma~\ref{C}. Each element $u\in I_A$ can be uniquely represented as follows: $u=u_1u_2$ where $u_1\in Q_A$ and $u_2\in \op{Path}(E_A)\subset \op{Path}(E)$. Observe that $u_1$ and $u_2$ could be equal to some vertex $e\in A$. Define the map $f:J_A^0\rightarrow B_{Q_A}^0(G(E_A))$ by the following way: $f(0)=0$ and for each non-zero element $uv^{-1}=u_1u_2(v_1v_2)^{-1}\in G(E)$ where $u_1,v_1\in Q_A$ and $u_2,v_2\in\op{Path}(E_A)\subset \op{Path}(E)$ put $f(uv^{-1})=(u_1,u_2v_2^{-1},v_1)$. The injectivity of the map $f$ is straightforward.
Next we show that the map $f$ is a homomorphism. Fix any elements $ab^{-1},cd^{-1}\in J_A$. Following the main idea of the proof we can uniquely represent elements $a,b,c,d\in I_A$ as follows: $a=a_1a_2, b=b_1b_2, c=c_1c_2$ and $d=d_1d_2$ where $a_1,b_1,c_1,d_1\in Q_A$ and $a_2,b_2,c_2,d_2\in G(E_A)$. 
There are three cases to consider:
\begin{itemize}
\item[$(1)$] there exists $w\in\op{Path}(E)$ such that $ab^{-1}\cdot cd^{-1}=awd^{-1}$, i.e., $c=bw$;
\item[$(2)$] there exists $w\in\op{Path}(E)$ such that $ab^{-1}\cdot cd^{-1}=a(dw)^{-1}$, i.e., $b=cw$;
\item[$(3)$] $ab^{-1}\cdot cd^{-1}=0$.
\end{itemize}
Consider case $1$. Observe that $s(w)=r(b)\in A$ and $r(w)=r(c)\in A$. By Lemma~\ref{C}, $w\in \op{Path}(E_A)$ which implies that $c_1=b_1$ and $c_2=b_2w$. Hence
\begin{equation*}
\begin{split}
&f(ab^{-1})\cdot f(cd^{-1})=(a_1,a_2b_2^{-1},b_1)\cdot (b_1,b_2wd_2^{-1},d_1)=(a_1,a_2b_2^{-1}\cdot b_2wd_2^{-1}, d_1)=\\
&=(a_1,a_2wd_2^{-1}, d_1)=f(awd^{-1}).
\end{split}
\end{equation*}
Consider case $2$. Observe that $s(w)=r(c)\in A$ and $r(w)=r(b)\in A$. By Lemma~\ref{C}, $w\in \op{Path}(E_A)$ which implies that $c_1=b_1$ and $b_2=c_2w$. Similar calculations as in case $1$ show that $f(ab^{-1})\cdot f(cd^{-1})=f(a(dw)^{-1})$.

Consider case $3$. Observe that neither $b=cw$ nor $c=bw$. Then one of the following two subcases holds:
\begin{itemize}
\item[$(3.1)$] $b_1\neq c_1$;
\item[$(3.2)$] $b_1=c_1$, but for any $w\in\op{Path}(E)$ neither $b_2=c_2w$ nor $c_2=b_2w$.
\end{itemize}
Consider subcase $3.1$. Then
$$
f(ab^{-1})\cdot f(cd^{-1})=(a_1,a_2b_2^{-1},b_1)\cdot (c_1,c_2d_2^{-1},d_1)=0=f(0).
$$
Consider subcase $3.2$. Then
$$f(ab^{-1})\cdot f(cd^{-1})=(a_1,a_2b_2^{-1},b_1)\cdot (b_1,c_2d_2^{-1},d_1)=(a_1,a_2b_2^{-1}\cdot c_2d_2^{-1}, d_1)=(a_1,0, d_1)=0=f(0).$$
Hence the map $f$ is an isomorphic embedding of the semigroup $J_A^0$ into $B_{Q_A}^0(G(E_A))$.
\end{proof}

\section{A global structure of graph inverse semigroups}
By $\mathcal{A}$ we denote the set of all strongly connected components of a graph $E$. The set $\mathcal{A}$ admits a natural partial order $\leq$: for each $X,Y\in \mathcal{A}$, $X\leq Y$ iff there exists a path $u\in\op{Path}(E)$ such that $s(u)\in Y$ and $r(u)\in X$.
\begin{theorem}\label{charJ}
For any graph $E$ the following statements hold:
\begin{itemize}
\item[$(1)$] $G(E)=\cup_{X\in \mathcal{A}}J_X^0$;
\item[$(2)$] $J_X^0$ is isomorphic to a subsemigroup of $B_{Q_X}^0(G(E_X))$, for each $X\in\mathcal{A}$.
\item[$(3)$] $J_X^0\cap J_Y^0=\{0\}$ for each distinct elements $X,Y\in \mathcal{A}$;
\item[$(4)$] if $X\leq Y$ then $J_X^0\cdot J_Y^0\cup J_Y^0\cdot J_X^0\subseteq J_X^0$;
\item[$(5)$] if $X\nleq Y$ and $Y\nleq X$ then $J_X^0\cdot J_Y^0\cup J_Y^0\cdot J_X^0\subseteq \{0\}$.
\end{itemize}
\end{theorem}

\begin{proof}
Statements $1$ and $3$ follows from the fact that $\mathcal{J}$ is an equivalence relation. Statement $2$ follows from Theorem~\ref{JJ}.

Consider statement $4$. Assume that $X,Y\in \mathcal{A}$ and $X\leq Y$. Fix any elements $ab^{-1}\in J_X^0$ and $cd^{-1}\in J_Y^0$. Observe that the case $ab^{-1}\cdot cd^{-1}=0$ is trivial, because $0\in J_X^0$. Suppose that $ab^{-1}\cdot cd^{-1}\neq 0$. In this case there exists a path $w$ such that either $c=bw$ or $b=cw$. If $b=cw$ then $ab^{-1}\cdot cd^{-1}=a(dw)^{-1}$ and $r(dw)=r(w)=r(b)=r(a)\in X$. Hence $a(dw)^{-1}\in J_X$.
If $c=bw$ then $s(w)=r(b)\in X$ and $r(w)=r(c)\in Y$ which implies that $Y\leq X$. Since the order $\leq$ is antisymmetric we obtain that $X=Y$. Hence Corollary~\ref{J} provides that $ab^{-1}\cdot cd^{-1}\in J_X$.

Consider statement $5$. Assume that $X\nleq Y$ and $Y\nleq X$.
Fix any elements $ab^{-1}\in J_X^0$ and $cd^{-1}\in J_Y^0$. We claim that neither $b$ is a prefix of $c$ nor $c$ is a prefix of $b$. Indeed, if $b$ is a prefix of $c$, i.e., $c=bw$ for some path $w$. Then $s(w)=r(b)\in X$ and $r(w)=r(c)\in Y$ witnessing that $Y \leq X$ which contradicts to the assumption. If $c$ is a prefix of $b$, i.e., $b=cw$ for some path $w$. Then $s(w)=r(c)\in Y$ and $r(w)=r(b)\in X$ witnessing that $X \leq Y$ which contradicts to the assumption.
Hence $ab^{-1}\cdot cd^{-1}=0$.
\end{proof}

The proof of the following lemma follows from the definition of Green's relations $\mathcal{D}$ and $\mathcal{J}$.

\begin{lemma}\label{JD}
For a graph inverse semigroup $G(E)$ the following conditions are equivalent:
\begin{itemize}
\item[$(1)$] relations $\mathcal{J}$ and $\mathcal{D}$ coincide on $G(E)$;
\item[$(2)$] graph $E$ is acyclic.
\end{itemize}
\end{lemma}

Now we apply our results to graph inverse semigroups over acyclic graphs. Observe that each strongly connected component of an acyclic graph $E$ coincides with some vertex $e\in E^0$.
Hence each acyclic graph $E$ admits a natural partial order $\leq$ on the set $E^0$. For each $e,f\in E^0$, $e\leq f$ iff there exists a path $u$ such that $s(u)=f$ and $r(u)=e$. The following theorem describes the structure of graph inverse semigroups over acyclic graphs.

\begin{theorem}\label{char}
Let $E$ be an acyclic graph. Then the following statements hold:
\begin{itemize}
\item[$(1)$] $G(E)=\cup_{e\in E^0}D_e^0$;
\item[$(2)$] $D_e^0$ is isomorphic to the semigroup of $I_e{\times}I_e$-matrix units $\mathcal{B}^0_{I_e}$, for each vertex $e\in E^0$;
\item[$(3)$] $D_e^0\cap D_f^0=\{0\}$, for each distinct vertices $e,f\in E^0$;
\item[$(4)$] If $e\leq f$ then $D_e^0\cdot D_f^0\cup D_f^0\cdot D_e^0\subseteq D_e^0$;
\item[$(5)$] If $e\nleq f$ and $f\nleq e$ then $D_e^0\cdot D_f^0\cup D_f^0\cdot D_e^0= \{0\}$.
\end{itemize}
\end{theorem}

\begin{proof}
Statements $1$ and $3$ are obvious. Statement $2$ follows from Corollary~\ref{matrix}.
Statement $4$ (resp., $5$) follows from Lemma~\ref{JD} and statement $4$ (resp., $5$) of Theorem~\ref{charJ}.
\end{proof}

\section*{Acknowledgements}
The author acknowledges the referee for his comments and suggestions.

\end{document}